\documentclass[12pt,notitlepage]{article}
\usepackage{amsmath}
\usepackage{amssymb}
\usepackage{amsfonts}
\usepackage{amsthm}
\usepackage{xcolor}

\usepackage{authblk}

\newtheorem{theo}{Theorem}[section]

\newtheorem{lemm}[theo]{Lemma}

\newtheorem{defi}[theo]{Definition}
\newtheorem{rema}[theo]{Remark}

\newcommand{\N}{\mathbb{N}}

\allowdisplaybreaks

\makeatletter
\newcommand{\subjclass}[2][2010]{%
  \let\@oldtitle\@title%
  \gdef\@title{\@oldtitle\footnotetext{#1 \emph{Mathematics subject classification.} #2}}%
}
\newcommand{\keywords}[1]{%
  \let\@@oldtitle\@title%
  \gdef\@title{\@@oldtitle\footnotetext{\emph{Key words and phrases.} #1.}}%
}
\makeatother

\begin{document}
\title{A Menon-type Identity concerning Dirichlet characters and a generalization of the gcd function}
 

\author[1]{Arya Chandran}

\author[2]{K Vishnu Namboothiri\thanks{Corresponding Author}}

\author[3]{Neha Elizabeth Thomas}
\affil[1]{Department of Mathematics, University College, Thiruvananthapuram, Kerala - 695034, India, Email : \texttt{aryavinayachandran@gmail.com}}

\affil[2]{Department of Mathematics, Government College, Ambalapuzha, Kerala - 688561, INDIA\\Department of Collegiate Education, Government of Kerala,  Email : \texttt{kvnamboothiri@gmail.com}}
\affil[3]{Department of Mathematics, SD College, Alappuzha, Kerala - 688003, India, Email : \texttt{nehathomas2009@gmail.com}}
\date{}
	\maketitle
	\begin{abstract}
		 Menon's identity is a classical identity involving gcd sums and the Euler totient function $\phi$. In a recent paper, Zhao and Cao derived the Menon-type identity $\sum\limits_{\substack{k=1}}^{n}(k-1,n)\chi(k) = \phi(n)\tau(\frac{n}{d})$, where $\chi$ is a Dirichlet character mod $n$ with conductor $d$. We derive an identity similar to this replacing gcd with a generalization it. We also show that some of the arguments used in the derivation of Zhao-Cao identity can be improved if one uses the method we employ here.
	\end{abstract}

\section{Introduction}
  Menon's identity that originally appeared in \cite{menon1965sum} is a gcd sum turning out to be equal to a product of the Euler totient function $\phi$ and the number of divisors function $\tau$. If $(m,n)$ denotes the gcd of $m$ and $n$, the identity states that
\begin{align}\label{menons-identity}
\sum\limits_{\substack{m=1\\(m,n)=1}}^n (m-1,n)=\phi(n)\tau(n).
\end{align}

 This identity was generalized by several authors in various directions. For example, B. Sury \cite{sury2009some} derived the following Menon-type identity $$\sum\limits_{\substack{1\leq m_1, m_2,\ldots,m_s\leq n\\(m_1,n)=1}}(m_1-1, m_2,\ldots,m_s,n)=\phi(n)\sigma_{s-1}(n)$$ where $\sigma_s(n)=\sum\limits_{d|n}d^s$ using properties of group actions. When $s=1$, this becomes the Menon's identity. Zhao and Cao \cite{zhao2017another} recently derived another  Menon-type identity 
 \begin{align}\label{zhao}
 \sum\limits_{\substack{k=1}}^{n}(k-1,n)\chi(k) = \phi(n)\tau(\frac{n}{d})
 \end{align}
 where $\chi$ is the Dirichlet character mod $n$ with conductor $d$. When $\chi$ is the principal character mod $n$, this identity reduces to the Menon's identity. A generalization of this Zhao-Cao identity involving even functions mod $n$ was derived by L. T{\'o}th  in \cite{toth2018menon}.

 For positive integers $a, b$ and $s$, E. Cohen \cite{cohen1956some} suggested a generalization of the gcd function which we denote in this paper by $(a,b)_s$ (see next section for the definition of this function). In \cite{chandran2020generalization}, the authors of this paper proposed a generalization to the Menon's identity which was obtained by replacing the gcd function with $(a,b)_s$.  Various other  generalizations of the Menon's identity were provided by many authors. See, for example \cite{haukkanen2005menon}, \cite{haukkanen1996generalization}, \cite{ramaiah1978arithmetical}, \cite{toth2011menon} and the more recent papers \cite{haukkanen2019menon} and \cite{toth2019short}.

The Klee's function $\Phi_s$ is a natural generalization of the Euler totient function. The generalized divisor function $\tau_s$ defined in \cite{chandran2020generalization} generalizes the usual divisor function $\tau$ (see next section for the definitions of these generalizations). A natural question arising is if the gcd function in the Zhao-Cao identity (\ref{zhao}) is replaced with the generalized gcd function suggested by E. Cohen, what could be the possible change that can happen to this identity. We propose here a Menon-type identity  modifying the identity (\ref{zhao}) replacing the gcd function appearing in (\ref{zhao}) with generalized gcd function. Our techniques closely follow the style of arguments appearing in \cite{zhao2017another}. The main results we propose in this paper are the following.
\begin{theo}\label{theo1}
	Let $s,n \in \N$ and $\chi$ be a primitive Dirichlet character mod $n$, where $n$ is the $s^{th}$ power of some natural number. Then 
	\begin{displaymath}
	\sum\limits_{\substack{k=1\\(k,n)_s = 1}}^{n} (k-1,n)_s \chi(k) = \Phi_s(n).
	\end{displaymath}
	
\end{theo}
\begin{theo}\label{theo2}
	Let $\chi$ be a Dirichlet character mod $n$, where $n = m^{qs}$, $m,q,s \in \mathbb{N}$. If $d =m^{ts}$, $1 \leq t \leq q$ is the conductor of $\chi$ then 
	\begin{displaymath}
	\sum\limits_{\substack{k=1\\(k,n)_s=1}}^{n}(k-1,n)_s \chi(k) = \Phi_s(n)\tau_s(n/d).
	\end{displaymath}
\end{theo}
\section{Notations and basic results} 
Most of the notations, functions and identities we use in this paper are standard and their definitions can be found in \cite{tom1976introduction}.  We give below the definitions of some other less popular terms and functions which we use in this paper.
\begin{defi}[\cite{cohen1956some}]
For positive integers $a,b$ and $s$ the generalized gcd of $a$ and $b$ denoted by $(a,b)_s$ is defined to be the largest $l^s$ (where $l \in \N$) dividing both $a$ and $b$.  
\end{defi}
$(a, b)_1$ is thus the usual gcd of $a$ and $b$. Like the gcd function, $(a,b)_s = (b,a)_s$.

The next statement is elementary and can be proved easily. We state it without proof.
\begin{lemm}
 $(a,b)_{s}$ is multiplicative in first variable.
 \end{lemm}
 
 It can be further observed that  $(a,b)_{s}$ is not completely multiplicative as a single variable function of $a$. Also, it is not multiplicative in $s$.

\begin{defi}
If $(a,b)_s = 1$, then we say that $a$ and $b$ are relatively $s$-prime to each other. 
\end{defi}

\begin{defi}[\cite{klee1948generalization}]
The Klee's function $\Phi_s(n)$ is defined as the cardinality of the set $\{m\in \N : 1\leq m \leq n,(m,n)_s=1\}$. 
\end{defi}
Thus $\Phi_s(n)$ denotes the number of positive integers $\leq n$ that are relatively $s$-prime to $n$. Various properties satisfied by $\Phi_s(n)$ are listed in \cite[Section 2]{chandran2020generalization}. 

If $M$ is a complete residue system mod $n$, then the subset of elements from $M$ that are relatively $s$-prime to $n$ is called an $s$-reduced system. Further, if $M$ is a subset of  $\{a: 0\leq a < n\}$ then the $s$-reduced system is called a minimal $s$-reduced residue system (mod $n$).  
\begin{defi}
For natural numbers $n$ and $s$, by $\tau_s(n)$ we mean the number of $l^s$  dividing $n$ where $l\in \N$. 
\end{defi}
It was observed in \cite{chandran2020generalization} that $\Phi_s(n)$ and $\tau_s(n)$ are multiplicative in $n$.

The following lemma is essential to prove one of the main results that we propose in this paper.
 	
 \begin{lemm}\label{l2}\cite[Lemma 3]{cohen1956some} 
Let $A = \{m \mid 1 \leq m \leq n \text{ and } (m,n)_s = 1\}$ and let $d>0$ be any $s$\textsuperscript{th} power divisor of $n$. Then $A$ is the union of $\frac{\Phi_s(n)}{\Phi_s(d)}$ disjoint sets each of which is an $s$-reduced residue system mod $d$.\end{lemm}

 \section{Proofs of the main results} 

We here provide proofs of the claims we made in the first section. To prove Theorem \ref{theo1}, we need the following lemma. 
\begin{lemm}\label{l1}
	Let $s,n \in \mathbb{N}$ and $\chi$ be a primitive Dirichlet character mod $p^n$, where $p$ is prime and $n $ is a multiple of $s$.  If $m$ is a multiple of $s$ such that $s\leq m<n$, then  $$\sum\limits_{\substack{k = 1\\(k,p^{n-m})_s=1}}^{p^{n-m}}  \chi(kp^m+1) =\begin{cases} -1, \quad m = n-s\\
	0, \quad\text{ otherwise.}
	\end{cases}$$\\
\end{lemm}
\begin{proof}
	By the conditions imposed on $s, m $ and $n$ we see that $n\neq s$. Suppose $m = n-s$. Since $p^{n-s}$ is not an induced modulus for $\chi$, there exists an integer $b$, $1\leq b <p^s$ with $(bp^{n-s}+1,p^n) = 1$ and $bp^{n-s}+1\equiv 1 (\text{mod } p^{n-s})$, but $\chi(bp^{n-s}+1) \neq 1$. So
	\begin{align*}
	\chi(bp^{n-s}+1)  \sum\limits_{\substack{{k = 0}}}^{p^{s}-1} \chi(kp^{n-s}+1) 
	& = \sum\limits_{\substack{{k = 0}}}^{p^{s}-1} \chi(kbp^{2n-2s}+bp^{n-s}+kp^{n-s}+1)\\ 
	& = \sum\limits_{\substack{{k = 0}}}^{p^{s}-1} \chi((k+b)p^{n-s}+1) \\
	& = \sum\limits_{\substack{{k = 0}}}^{p^{s}-1} \chi(kp^{n-s}+1). 
	\end{align*}
	Hence $\sum\limits_{\substack{{k = 0}}}^{p^{s}-1} \chi(kp^{n-s}+1) = 0$ and so $ \sum\limits_{\substack{{k = 1}}}^{p^{s}} \chi(kp^{n-s}+1) = 0$. It follows that
	\begin{align*}
	\sum\limits_{\substack{k =1\\(k,p^{s})_s=1}}^{p^{s}} \chi(kp^{n-s}+1) 
	& =  \sum\limits_{\substack{k =1}}^{p^{s}} \chi(kp^{n-s}+1)-  \sum\limits_{\substack{k =1\\(k,p^{s})_s\neq 1}}^{p^{s}} \chi(kp^{n-s}+1)\\
	& = - \sum\limits_{\substack{k =1\\(k,p^{s})_s\neq 1}}^{p^{s}} \chi(kp^{n-s}+1)\\
	&= -\chi(kp^n+1)\\
	&=-\chi(1) \\
	&= -1
	\end{align*}
	Next we consider the case $m\neq n-s$. As in the previous case. 
	\begin{align*}
	\chi(bp^{n-s}+1)  \sum\limits_{\substack{{k = 1}\\{(k,p^{n-m})_s=1}}}^{p^{n-m}} \chi(kp^{m}+1) & =  \sum\limits_{\substack{{k = 1}\\{(k,p^{n-m})_s=1}}}^{p^{n-m}} \chi(bkp^{m}p^{n-s}+kp^{m}+bp^{n-s}+1)\\ & =  \sum\limits_{\substack{{k = 1}\\{(k,p^{n-m})_s=1}}}^{p^{n-m}} \chi(kp^m+bp^{n-s}+1)
	\end{align*}
	We claim that $\{kp^m+bp^{n-s}+1: 1\leq k\leq p^{n-m}, (k,p^{n-m})_s=1\}$ is the same as the residue system $kp^m+1$ mod $p^n$. Suppose  $1 \leq k_1 \leq p^{n-m}$ and $(k_1,p^{n-m})_s=1$. If $c\equiv k_1p^m+bp^{n-s}+1\text{ (mod $p^n$)}$ for some integer $c$, then let $k_2\equiv k_1+bp^{n-s-m}\text{ (mod $p^{n-m}$)}$. Note that if $(k_2,p^{n-m})_s=p^{rs}$ for some prime $p$ and $1 \leq r \leq \frac{n-m}{s}$, then we have $p^s\mid k_2$, which implies $p^s\mid k_1+bp^{n-s-m}$. But in this case $s\leq m \leq n-2s$ and $p^s\mid p^{n-s-m}$ implying that $p^s \mid k_1$ which is not possible. Therefore $(k_2,p^{n-m})_s=1$ and also $1\leq k_2 \leq p^{n-m}$. Now we have $k_2p^m+1 \equiv c\text{ (mod $p^n$)}$. If $k_1p^m+bp^{n-s}+1={k_1}' p^m+bp^{n-s}+1 \text{ (mod $p^n$)}$ then $k_1\equiv k_1'\text{ (mod $p^{n-m}$)}$. Similarly if $k_2p^m+1 \equiv k_2' p^m+1\text{ (mod $p^n$)}$ then $k_2 \equiv k_2' \text{ (mod $p^{n-m}$)}$. Therefore both these residue systems consists of  $\Phi_s(p^{n-m})$ different elements, and so we get $\chi(bp^{n-s}+1)\sum\limits_{\substack{k=1\\(k,p^{n-m})_s=1}}^{p^{n-m}}\chi(kp^m+1) = \sum\limits_{\substack{k=1\\(k,p^{n-m})_s=1}}^{p^{n-m}}\chi(kp^m+1)$. This implies that $\sum\limits_{\substack{k=1\\(k,p^{n-m})_s=1}}^{p^{n-m}}\chi(kp^m+1) = 0$ which is what we required.
	\end{proof}
	
\begin{proof}[Proof of Theorem \ref{theo1}]
	Let $f(n)=\sum\limits_{\substack{k=1\\(k,n)_s = 1}}^{n} (k-1,n)_s \chi_n(k)$, where $\chi_n$ is some Dirichlet character mod $n$. For $r,t \in \N$, we have $f(rt) = \sum\limits_{\substack{k=1\\(k,rt)_s = 1}}^{rt} (k-1,rt)_s \chi_{rt}(k)$. Now we use the fact that if $(r,t)=1$ then the two sets $\{k \mid 1\leq k\leq rt, (k,rt)_s=1\}$ and $\{tk_1+rk_2 \mid 1\leq k_1\leq r, (k_1,r)_s=1, 1\leq k_2\leq t,(k_2,t)_s=1\}$ are the same. Note that $\chi$ mod $k$ can be factored uniquely as a product of the form $\chi_k = \chi_{k_1}\chi_{k_2}\cdots \chi_{k_r}$, where $k = k_1k_2\cdots k_r$ with $(k_i,k_j)=1$ if $i \neq j$. In particular,  if $\chi$  is primitive then each $\chi_{k_i}$ is primitive mod $k_i$. Since the generalized gcd function is multiplicative in the second variable, we get
	\begin{align*}
	f(rt) & = \sum\limits_{\substack{k_1=1\\(k_1,r)_s = 1}}^{r}\sum\limits_{\substack{k_2=1\\(k_2,t)_s = 1}}^{t} (tk_1+rk_2-1,r)_s (tk_1+rk_2-1,t)_s\\& \times \chi_r(tk_1+rk_2)\chi_t(tk_1+rk_2)\\ 
	& = \sum\limits_{\substack{k_1=1\\(k_1,r)_s = 1}}^{r}\sum\limits_{\substack{k_2=1\\(k_2,t)_s = 1}}^{t} (tk_1+rk_2-1,r)_s (tk_1+rk_2-1,t)_s \chi_r(tk_1)\chi_t(rk_2),
	\end{align*} Now we observe that $(tk_1+rk_2-1,r)_s = (tk_1-1,r)_s$ and  $(tk_1+rk_2-1,t)_s =  (rk_2-1,t)_s$. So 
	\begin{align*}
	f(rt) & = \sum\limits_{\substack{k_1=1\\(k_1,r)_s = 1}}^{r}\sum\limits_{\substack{k_2=1\\(k_2,t)_s = 1}}^{t} (tk_1-1,r)_s (rk_2-1,t)_s \chi_r(tk_1)\chi_t(rk_2)\\ 
	& = \sum\limits_{\substack{k_1=1\\(k_1,r)_s = 1}}^{r}(tk_1-1,r)_s\chi_r(tk_1)\sum\limits_{\substack{k_2=1\\(k_2,t)_s = 1}}^{t}  (rk_2-1,t)_s \chi_t(rk_2). 
	\end{align*}
	Since $(r,t) = 1$,
	\begin{align*}
	f(rt)  & = \sum\limits_{\substack{k_1=1\\(k_1,r)_s = 1}}^{r}(k_1-1,r)_s\chi_r(k_1)\sum\limits_{\substack{k_2=1\\(k_2,t)_s = 1}}^{t}  (k_2-1,t)_s \chi_t(k_2)\\ 
	& = f(r)f(t).
	\end{align*} 
	Thus $f$ is multiplicative and so we need to verify our claim only for prime powers $p^a$, where $a = qs$, $q \in \N$. Therefore 
	\begin{align*}
	f(p^a) &=  \sum\limits_{\substack{k=1\\(k,p^a)_s = 1}}^{p^a}(k-1,p^a)_s\chi_{p^a}(k)\\ 
	& =\sum\limits_{\substack{k=1}}^{p^a}(k-1,p^a)_s\chi_{p^a}(k)- \sum\limits_{\substack{k=1\\(k,p^a)_s\neq 1}}^{p^a}(k-1,p^a)_s\chi_{p^a}(k)\\ 
	&= \sum\limits_{\substack{k=1}}^{p^a}(k-1,p^a)_s\chi_{p^a}(k)\\
	&= \sum\limits_{\substack{k=1\\(k-1,p^a)_s \neq 1}}^{p^a}(k-1,p^a)_s\chi_{p^a}(k)+ \sum\limits_{\substack{k=1\\(k-1,p^a)_s = 1}}^{p^a}\chi_{p^a}(k)\\ 
	& = \sum\limits_{\substack{k=1\\(k-1,p^a)_s \neq 1}}^{p^a}(k-1,p^a)_s\chi_{p^a}(k)+\sum\limits_{\substack{k=1\\}}^{p^a}\chi_{p^a}(k)- \sum\limits_{\substack{k=1\\(k-1,p^a)_s \neq 1}}^{p^a}\chi_{p^a}(k)\\ 
	& = \sum\limits_{\substack{k=1\\(k-1,p^a)_s \neq 1}}^{p^a}(k-1,p^a)_s\chi_{p^a}(k)- \sum\limits_{\substack{k=1\\(k-1,p^a)_s \neq 1}}^{p^a}\chi_{p^a}(k)\\ 
	& = \sum\limits_{\substack{k=1\\(k-1,p^a)_s \neq 1}}^{p^a}((k-1,p^a)_s-1)\chi_{p^a}(k)\\ 
	& = \sum\limits_{\substack{t=1}}^{q} \sum\limits_{\substack{k=1\\(k-1,p^a)_s = p^{ts} }}^{p^a}(p^{ts}-1)\chi_{p^a}(k)\\ 
	& = \sum\limits_{\substack{k=1\\(k-1,p^a)_s = p^{a} }}^{p^a}(p^{a}-1)\chi_{p^a}(k)+\sum\limits_{\substack{t=1}}^{q-1} \sum\limits_{\substack{k=1\\(k-1,p^a)_s = p^{ts} }}^{p^a}(p^{ts}-1)\chi_{p^a}(k)\\ 
	& = (p^a-1)+\sum\limits_{\substack{t=1}}^{q-1}(p^{ts}-1)\sum\limits_{\substack{k=1\\(k-1,p^a)_s = p^{ts} }}^{p^a}\chi_{p^a}(k). 
	\end{align*}
	We need to compute the sum $\sum\limits_{\substack{k=1\\(k-1,p^a)_s = p^{ts} }}^{p^a}\chi_{p^a}(k)$. We have $\sum\limits_{\substack{k=1\\(k-1,p^a)_s = p^{ts} }}^{p^a}\chi_{p^a}(k) =\sum\limits_{\substack{k=1\\(k,p^a)_s = p^{ts} }}^{p^a}\chi_{p^a}(k+1)$. To evaluate this, for a fixed prime power $p^{ts}$ we take the sum over all those $k$ in the range $1\leq k \leq p^a$ where $(k,p^a)_s=p^{ts}$. If we write $k = jp^{ts}$ then $1\leq k \leq p^a$ and $(k,p^a)_s=1$ if and only if $1 \leq j \leq p^{a-ts}$ and $(j,p^{a-ts})_s=1$. Then the last sum can be re-written as $\sum\limits_{\substack{k=1\\(k,p^a)_s = p^{ts} }}^{p^a}\chi_{p^a}(k+1) = \sum\limits_{\substack{j=1\\(j,p^{a-ts})_s = 1 }}^{p^{a-ts}}\chi_{p^{a}}(jp^{ts}+1)$ and $f(p^a)= (p^a-1)+\sum\limits_{\substack{t=1}}^{q-1}(p^{ts}-1) \sum\limits_{\substack{j=1\\(j,p^{a-ts})_s = 1 }}^{p^{a-ts}}\chi_{p^{a}}(jp^{ts}+1)$. By Lemma \ref{l1}, we obtain $\sum\limits_{\substack{j=1\\(j,p^{a-ts})_s = 1 }}^{p^{a-ts}}\chi_{p^{a}}(jp^{ts}+1)=\begin{cases} -1 \text{ if } t = q-1\\
	0 \text{ otherwise}
	\end{cases}$.  Then
	\begin{align*} 
	f(p^a) & = p^a-1+(p^{(q-1)s}-1)(-1)\\ 
	& =p^a-p^{qs-s}\\ 
	& =p^a-p^{a-s}\\ 
	& = \Phi_s(p^a),\text{which concludes the proof.}
	\end{align*}
	
\end{proof}
The above theorem reduces to Theorem 1.1 in \cite{zhao2017another} when $s=1$. We would like to further remark that Theorem 1.1 in \cite{zhao2017another} was proved using Lemma 2.1 and Lemma 2.2 in \cite{zhao2017another}. If one employs the technique we used above, only \cite[Lemma 2.1]{zhao2017another} is required to prove \cite[Theorem 1.1]{zhao2017another}.

To prove Theorem \ref{theo2}, we require the following two lemmas. First lemma  generalizes  \cite[Lemma 2.4]{zhao2017another}.
\begin{lemm}\label{l3}
	Let $s,n \in \N $ and $\chi $ be a Dirichlet character mod $p^n$, where $n = qs$ for some $q\in \N$. Let $p^l$ be the conductor of $\chi$ where $l = rs$ for some $r \in \N$ and $1 \leq r \leq q$ . If $m$ is a multiple of $s$ such that $s\leq m < n$, we have $\sum\limits_{\substack{k=1\\(k,p^{n-m})_s = 1}}^{p^{n-m}} \chi(kp^m+1) = \begin{cases}\Phi_s(p^{n-m}), \text{ if } l \leq m < n\\
	-p^{n-l}, \text{ if } m = l-s\\ 0, \text{ if } s\leq m < l-s.
	\end{cases}$	
\end{lemm}
\begin{proof}
	First we consider the case $ l \leq m < n$. We have
	\begin{align*} 
	\sum\limits_{\substack{k=1\\(k,p^{n-m})_s = 1}}^{p^{n-m}} \chi(kp^m+1)  & = \sum\limits_{\substack{k=1\\(k,p^{n-m})_s = 1}}^{p^{n-m}} \chi(1)\\ & = \sum\limits_{\substack{k=1\\(k,p^{n-m})_s = 1}}^{p^{n-m}} 1\\ & = \Phi_s(p^{n-m}). 
	\end{align*}
	Next we move on to the case $s\leq m \leq l-s$. Note that every Dirichlet character $\chi$ mod $k$ can be expressed as a product of the form $\chi(n) = \psi(n)\chi_1(n)$ for all $n$, where $\psi$ is a primitive character modulo conductor of $\chi$ and $\chi_1$ is the principal character mod $n$. Then $\sum\limits_{\substack{k=1\\(k,p^{n-m})_s = 1}}^{p^{n-m}} \chi(kp^m+1) = \sum\limits_{\substack{k=1\\(k,p^{n-m})_s = 1}}^{p^{n-m}} \psi(kp^m+1) \chi_1(kp^m+1)$, where $\psi$ is the primitive character mod conductor of $\chi$ and $\chi_1$ is the principal character mod $p^n$. Since $s\leq m \leq l-s$, $(kp^m+1,p^n)=1$, using  Lemma \ref{l2} and Lemma \ref{l1} we get 
	\begin{align*} 
	\sum\limits_{\substack{k=1\\(k,p^{n-m})_s = 1}}^{p^{n-m}} \chi(kp^m+1)  
	& = \sum\limits_{\substack{k=1\\(k,p^{n-m})_s = 1}}^{p^{n-m}} \psi(kp^m+1)\\ 
	& = \frac{\Phi_s(p^{n-m})}{\Phi_s(p^{l-m})} \sum\limits_{\substack{k=1\\(k,p^{l-m})_s = 1}}^{p^{l-m}}\psi(kp^m+1)\\ 
	& = p^{n-l}\sum\limits_{\substack{k=1\\(k,p^{l-m})_s = 1}}^{p^{l-m}}\psi(kp^m+1)\\ 
	& =\begin{cases} -p^{n-l} \text{ if } m = l-s\\
	0 \text{ if } s\leq  m < l-s, 
	\end{cases}\\ \text{which completes the proof}. 
	\end{align*}
\end{proof}
Next we prove a lemma, which is key to the proof of Theorem \ref{theo2}.
\begin{lemm}\label{l4}
	Let $s,a \in N $ and $\chi $ be a Dirichlet character mod $p^a$, where $a = qs$ for some $q \in \N$. If $p^{rs}$ is the conductor of $\chi$ where $r\in \N$ and $1 \leq r \leq q$ , we have $\sum\limits_{\substack{k=1\\(k,p^{a})_s = 1}}^{p^{a}} (k-1,p^a)_s \chi(k) = (q-r+1) \Phi_s(p^a).$
\end{lemm}
\begin{proof}
	We prove the lemma case by case.\\
	Case $1$. $r=1$\\ In this case $p^s$ is the conductor of $\chi$. From the proof of Theorem \ref{theo1}, we have  $f(p^a)= (p^a-1)+\sum\limits_{\substack{t=1}}^{q-1}(p^{ts}-1) \sum\limits_{\substack{j=1\\(j,p^{a-ts})_s = 1 }}^{p^{a-ts}}\chi_{p^{a}}(jp^{ts}+1)$. Using Lemma \ref{l3},
	\begin{align*} 
	\sum\limits_{\substack{k=1\\(k,p^{a})_s = 1}}^{p^{a}} (k-1,p^a)_s \chi(k)  
	& = p^a-1+\sum\limits_{\substack{t=1}}^{q-1}(p^{ts}-1)\sum\limits_{\substack{j=1\\(j,p^{a-ts})_s = 1}}^{p^{a-ts}} \chi_{p^a}(jp^{ts}+1)\\
	& = p^a-1+\sum\limits_{\substack{t=1}}^{q-1}(p^{ts}-1)\Phi_s(p^{a-ts})\\
	& = p^a-1+\sum\limits_{\substack{t=1}}^{q-1}(p^{ts}-1)p^{a-ts}(1-\frac{1}{p^{s}})\\ 
	& = p^a-1+\sum\limits_{\substack{t=1}}^{q-1}(p^a-p^{a-ts})(1-p^{-s})\\ 
	& = p^a-1+\sum\limits_{\substack{t=1}}^{q-1}(p^a-p^{a-s}-p^{a-ts}+p^{a-(t+1)s})\\  
	& =p^a-1+\sum\limits_{\substack{t=1}}^{q-1}(p^a-p^{a-s})+\sum\limits_{\substack{t=1}}^{q-1}(p^{a-(t+1)s}-p^{a-ts})\\  
	& =p^a-1+(p^a-p^{a-s})(q-1)+(p^{a-qs}-p^{a-s})\\ 
	& =p^a-1+(p^a-p^{a-s})(q-1)+1-p^{a-s}\\ 
	& = q(p^a-p^{a-s})\\ 
	& =q \Phi_s(p^a)
	\end{align*}
	Case $2$. $r=q$\\ 
	In this case $\chi$ is the primitive character mod $p^a$. The claim immediately follows from Theorem \ref{theo1}.
	\\
	Case $3$. $2 \leq r \leq q-1$\\
	As in the first case we have  $$f(p^a)= (p^a-1)+\sum\limits_{\substack{t=1}}^{q-1}(p^{ts}-1) \sum\limits_{\substack{j=1\\(j,p^{a-ts})_s = 1 }}^{p^{a-ts}}\chi_{p^{a}}(jp^{ts}+1).$$ By Lemma \ref{l3}, we get $\sum\limits_{\substack{j=1\\(j,p^{a-ts})_s = 1}}^{p^{a-ts}} \chi(jp^{ts}+1) = \begin{cases}\Phi_s(p^{a-ts}), \text{ if }r \leq t < q\\
	-p^{a-rs}, \text{ if } t = r-1\\ 0, \text{ if } 1\leq t< r-1.
	\end{cases}$
    Now
	\begin{align*} 
	f(p^a) & = p^a-1+\sum\limits_{\substack{t=1}}^{q-1}(p^{ts}-1)\sum\limits_{\substack{j=1\\(j,p^{a-ts})_s = 1}}^{p^{a-ts}} \chi_{p^a}(jp^{ts}+1)\\ & = p^a-1+\sum\limits_{\substack{t=1}}^{r-2}(p^{ts}-1)\sum\limits_{\substack{j=1\\(k,p^{a-ts})_s = 1}}^{p^{a-ts}} \chi(jp^{ts}+1)+ (p^{(r-1)s}-1)(-p^{a-rs})\\
	& \text{ }+\sum\limits_{\substack{t=r}}^{q-1}(p^{ts}-1)\sum\limits_{\substack{j=1\\(k,p^{a-ts})_s = 1}}^{p^{a-ts}} \chi(jp^{ts}+1) \\ 
	& = p^a-1-(p^{rs-s}-1)p^{a-rs}+\sum\limits_{\substack{t=r}}^{q-1}(p^{ts}-1)\Phi_s(p^{a-ts})\\ & = p^a-1-(p^{rs-s}-1)p^{a-rs}+\sum\limits_{\substack{t=r}}^{q-1}(p^{ts}-1)p^{a-ts}(1-\frac{1}{p^s})\\ 
	& =p^a-1-p^{a-s}+p^{a-rs}+\sum\limits_{\substack{t=r}}^{q-1}(p^{ts}-1)p^{a-ts}(1-p^{-s}) \\  
	& =p^a-1-p^{a-s}+p^{a-rs}+\sum\limits_{\substack{t=r}}^{q-1}(p^{a}-p^{a-s}-p^{a-ts}+p^{a-(t+1)s})\\  
	&=p^a-1-p^{a-s}+p^{a-rs}+\sum\limits_{\substack{t=r}}^{q-1}(p^{a}-p^{a-s})+\sum\limits_{\substack{t=r}}^{q-1}(p^{a-(t+1)s}-p^{a-ts})\\
	& = p^a-1-p^{a-s}+p^{a-rs}+(p^{a}-p^{a-s})(q-r)+p^{a-qs}-p^{a-rs}\\
	& =p^a-1-p^{a-s}+(p^a-p^{a-s})(q-r)+1 \\
	& =(q-r+1)(p^a-p^{a-s})\\
	& = (q-r+1)\Phi_s(p^a).
	\end{align*} 
\end{proof}

\begin{rema}
 Lemma 3.4 is very much similar to \cite[Lemma 3.1]{zhao2017another}. The identity in \cite[Lemma 3.1]{zhao2017another} reduces to the Menon's identity when $\chi$ is a principal character. But because of the assumptions in the lemma above,  $\chi$ cannot be taken as the principal character and so this lemma cannot be strictly taken as a generalization of \cite[Lemma 3.1]{zhao2017another}.
\end{rema}
Finally we prove Theorem \ref{theo2}, which is similar to Theorem 1.2 in \cite{zhao2017another}. But our conditions are more restrictive than those appearing in \cite[Theorem 1.2]{zhao2017another}. 

\begin{proof}[Proof of Theorem \ref{theo2}]
	We use the fact that if $n = p_1^{a_1}p_2^{a_2}\cdots p_r^{a_r}$ then $\chi_n = \chi_{p_1^{a_1}}\chi_{p_2^{a_2}}\cdots\chi_{p_r^{a_r}}$, where $\chi_n$ is the Dirichlet character mod $n$. Also if $g(\chi)$ denotes the conductor of $\chi$, then  $g(\chi_n) = g(\chi_{p_1^{a_1}})g(\chi_{p_2^{a_2}})\cdots g(\chi_{p_r^{a_r}})$. Let $n =p_1^{a_1s}p_2^{a_2s}\cdots p_r^{a_rs}$, $d =  p_1^{b_1s}p_2^{b_2s}\cdots p_r^{b_rs}$ where $1\leq b_i\leq a_i$. Now $f(n) = \sum\limits_{\substack{k=1\\(k,n)_s=1}}^{n}(k-1,n)_s \chi_n(k)$ is multiplicative. Therefore
	\begin{align*}
	\sum\limits_{\substack{k=1\\(k,n)_s=1}}^{n}(k-1,n)_s \chi(k)& = f(n)\\& =  f(p_1^{a_1s})f(p_2^{a_2s})\cdots f(p_r^{a_rs})\\& = \prod\limits_{\substack{i=1}}^{r}f(p_i^{a_is})\\& = \prod\limits_{\substack{i=1}}^{r}\sum\limits_{\substack{k=1\\(k,p_i^{a_is})_s=1}}^{p_i^{a_is}}(k-1,p_i^{a_is})_s \chi_{p_i^{a_is}}(k)
	\end{align*}
	Note that $p_1^{b_1s}p_2^{b_2s}\cdots p_r^{b_rs} = g(\chi_{p_1^{a_1s}})g(\chi_{p_2^{a_2s}})\cdots g(\chi_{p_r^{a_rs}})$. It is clear that $g(\chi_{p_i^{a_is}})=p_i^{b_is}$. Hence by Lemma \ref{l4} , 
	\begin{align*}
	\sum\limits_{\substack{k=1\\(k,n)_s=1}}^{n}(k-1,n)_s \chi(k)& =  \prod\limits_{\substack{i=1}}^{r}(a_i-b_i+1)\Phi_s(p_i^{a_is})\\& =\prod\limits_{\substack{i=1}}^{r}\tau_s(p_i^{(a_i-b_i)s})\Phi_s(p_i^{a_is})\\& = \Phi_s(n)\tau_s(\frac{n}{d}),\\ \text{which completes the proof.}
	\end{align*}
\end{proof}
\begin{rema}
 A strict generalization of Theorem 1.2 in \cite{zhao2017another} would have been $\sum\limits_{\substack{k=1\\(k,n)_s=1}}^{n}(k-1,n)_s \chi(k) = \Phi_s(n)\tau_s(n/d)$, where $\chi$ is a Dirichlet character mod $n$ with conductor $d$. But this identity cannot be derived. For example,  if we take $q = 1$, $s = 2$, $r=0$ and $p = 2$, the LHS of this identity  evaluates to $(0,4)_2+(2,4)_2=5$ where as the RHS gives $6$. 
 \end{rema}
In \cite{toth2018menon}, T{\'o}th  derived an identity similar to Menon's identity involving even functions mod $n$, M\"{o}bius function and the Euler totient function. Note that an arithmetical function is $n$-even if $f(k)=f((k,n))$. A concept similar to $n$-even function is $(n,s)$-even functions defined by McCarthy. An arithmetical function $f$ is $(n,s)$-even if $f(k)=f((k, n^s)_s)$  (see \cite{mccarthy1960generation} for details). Many of the properties of such functions were studied in \cite{namboothiri2019discrete}. We feel that T{\'o}th's results can be generalized to $(n,s)$-even functions and similar identities can be derived if one uses the results appearing in \cite{namboothiri2019discrete}. 
    
\section{Acknowledgements}
The first author thanks the University Grants Commission of India for providing financial support for carrying out research work through their Junior Research Fellowship (JRF) scheme. The third author thanks the Kerala State Council for Science,Technology and Environment, Thiruvananthapuram, Kerala, India for providing financial support  for carrying out research work.


\end{document}